\pdfoutput=1

\documentclass[11pt,a4paper,dvipsnames]{amsart}

\usepackage[utf8]{inputenc}
\usepackage[T1]{fontenc}
\usepackage[english]{babel}
\usepackage[a4paper,margin=1in]{geometry}
\usepackage{microtype,lmodern}
\usepackage{mathpazo}
\usepackage{euler}
\usepackage{amsmath,amssymb,amsthm,mathtools,mathrsfs,bm}
\usepackage{tikz-cd}
\usepackage{xcolor}
\usepackage{hyperref}
\hypersetup{hidelinks}
\usepackage[capitalize,nameinlink]{cleveref}
\numberwithin{equation}{section}

\theoremstyle{plain}
\newtheorem{theorem}{Theorem}[section]
\newtheorem{proposition}[theorem]{Proposition}
\newtheorem{lemma}[theorem]{Lemma}

\theoremstyle{definition}

\newtheorem{remark}[theorem]{Remark}

\newcommand{\cO}{\mathcal O}
\newcommand{\cE}{\mathcal E}
\newcommand{\cF}{\mathcal F}
\newcommand{\cH}{\mathcal H}
\newcommand{\cJ}{\mathcal J}
\newcommand{\cL}{\mathcal L}
\newcommand{\M}{\mathfrak M}
\newcommand{\SM}{\mathcal S\mathcal M}
\newcommand{\At}{\operatorname{At}}
\newcommand{\End}{\operatorname{End}}
\newcommand{\Hom}{\operatorname{Hom}}
\newcommand{\Pic}{\operatorname{Pic}}

\newcommand{\tr}{\operatorname{tr}}
\newcommand{\id}{\operatorname{id}}
\newcommand{\CC}{\mathbb C}
\newcommand{\FF}{\mathbb F}
\newcommand{\PP}{\mathbb P}
\newcommand{\ZZ}{\mathbb Z}
\newcommand{\dd}{\mathrm d}
\newcommand{\cHom}{\mathcal{H}om}

\title{Genus 4 Supermoduli Space Is Not Projected}

\author{Ron Donagi}
\address{Department of Mathematics, University of Pennsylvania,
David Rittenhouse Laboratories, 209 South 33rd Street,
Philadelphia, PA 19104-6395, USA}
\email{donagi@math.upenn.edu}

\author{Simone Noja}
\address{Dipartimento di Matematica, Universit\`a degli Studi di Bari
Aldo Moro, Via Edoardo Orabona 4, 70125 Bari, Italy}
\email{simone.noja@uniba.it}

\date{}

\begin{document}

\begin{abstract}
The supermoduli stack $\M_g$, parametrizing smooth unpunctured super Riemann surfaces of
genus $g$, is known to be non-projected for $g \ge 5$. We extend this result to the even-spin component of $\M_4$. This is done by constructing a morphism from a smooth projective curve to the even spin moduli stack $\SM_4^{+}$ and computing the pullback of the class obstructing the projectedness of $\M_4^{+}$. This pullback is non-zero, proving that $\M_4^{+}$ is non-projected.
\end{abstract}

\maketitle
\tableofcontents
\section{Introduction}

The supermoduli stack
$\M_g$, parametrizing smooth unpunctured super Riemann surfaces of genus $g$, is known to be non-projected for $g\geq 5$, cf. 
\cite[Theorem~1.1]{DonagiWittenNonprojected}.
In this paper we establish the corresponding result for the even component in genus four:

\begin{theorem}\label{thm:main}
The even component $\M_4^+$ of the smooth unpunctured genus-four
supermoduli stack is not projected.  Consequently, the full supermoduli
stack $\M_4=\M_4^+\sqcup\M_4^-$ is not projected.
\end{theorem}

As reviewed in Section~\ref{obstruction}, there is a standard obstruction class $\omega_2(S)$, cf.
\cite{Green}, whose non-vanishing implies non-projectedness of a supermanifold $S$.
As in \cite{DonagiWittenNonprojected}, we show that supermoduli is not projected by pulling back
$\omega_2(\M_4^+)$ along a morphism $\kappa:T\to\SM_4^+$ from a smooth projective curve,
and showing that this pullback does not vanish. 
Also as in \cite{DonagiWittenNonprojected}, the compact curve parametrizes certain branched covers of a fixed base curve $C$.
But here we take a shortcut. Much of the effort in \cite{DonagiWittenNonprojected} is dedicated to extending the covering construction over a non-projected superbase and then showing that its obstruction forces non-projectedness of the ambient $\M_g$.
That construction requires various constraints on the branched covers of $C$, e.g. the order of each ramification point must be odd. Here the spin family over the ordinary base $T$ already determines a split family of super Riemann surfaces, and hence a morphism $T\to\SM_g\to\M_g$. We avoid constructing a non-projected superbase and lifting the covering construction to it. The test curve $T$ itself is purely bosonic and therefore split; the class we compute is the pullback of the ambient obstruction $\omega_2(\M_4^+)$, whose non-vanishing implies that $\omega_2(\M_4^+)$ is non-zero.

One way to obtain a compact family of double covers of a curve $C$ is to choose a $C$ that admits a free action of a finite group $G$, taking the branch loci in $C$ to be $G$-orbits. The important point is that, due to the free action, these branch points never collide, so the branched covers remain in the interior of moduli.
However, there is a more general possibility: all we need is a curve $C$ which admits a fixed-point-free automorphism $\sigma$, and then we can take the branch loci to consist, e.g. of pairs $(p,\sigma p)$ for $p \in C$. This is more general, in that the group generated by $\sigma$ need not act freely; some powers of $\sigma$ may have fixed points.
The actual curve we use is described in Section~\ref{BD}; it is based on the curves constructed in
 \cite[\S2.1]{BryanDonagi}, with a typo corrected.  
Roughly speaking, it parametrizes genus-4 curves that are certain branched double covers of the particular genus 2 curve:
$$
 C:\quad y^2=x^6-1.
$$
After passing to a finite
$\acute{e}$tale
cover of its base, the resulting family of genus-four curves  carries an even theta
characteristic $\eta$ and underlies a family of super  Riemann surfaces.
The deck involution exchanging the sheets of the branched double covers has an order-four lift $J$ to the spin
line bundles $\eta$.  This $J$ acts on everything in sight, including the obstruction to projectedness.

Following \cite[\S3.2 and Proposition~3.1]{DonagiWittenAtiyah}, this obstruction can be interpreted as the class of a natural extension.
In Section \ref{Atiyah} we isolate the $J$-invariant part of this
extension, restrict its middle sheaf to the first infinitesimal
neighbourhood (the doubled diagonal), and identify the resulting pushout with
the direct image of a relative Atiyah class.
All the relevant data descend to the fixed
genus-two curve $C$.
There, the required non-vanishing reduces to an explicit \v Cech
calculation that we carry out in Section~\ref{calculation}.

\emph{Acknowledgments:} The authors are grateful to Mauricio Corr\^ea. His works 
\cite{CorreaNojaSplitting, CorreaNojaCompactified},
joint with one of us, have greatly influenced our thinking.
A work in progress, joint with him, examines the non-projectedness of $\M_3$.
We are also grateful to Sasha Polishchuk who ran the paper through
ChatGPT and shared with us its very helpful feedback.
The research of RD was partially supported 
by NSF grant DMS–2401422, Geometry and Strings; 
by NSF grant DMS–2244978, FRG: New birational invariants;
and by DFG--SFB 1624, \emph{Higher structures, moduli spaces and integrability}--506632645.

\section{The primary obstruction and the diagonal extension} \label{obstruction}

We first fix the obstruction-theoretic formulation.  The purpose of this
section is to identify a concrete extension on a family of spin curves whose
non-splitting implies the theorem.

Let $S=(M,\cO_S)$ be a complex supermanifold, or a smooth complex
Deligne--Mumford superstack interpreted in an $\acute{e}$tale atlas,
let $\cJ\subset\cO_S$ be the ideal generated by odd functions,
and let 

$$
 E^\vee=\cJ/\cJ^2.
$$
The supermanifold $S$ is \emph{projected} if the quotient
$q_{\mathrm{red}}:\cO_S\twoheadrightarrow\cO_M$ admits a section
$s:\cO_M\to\cO_S$ as a parity-preserving morphism of sheaves of
$\CC$-algebras, with
$q_{\mathrm{red}}\circ s=\id_{\cO_M}$.  Its primary obstruction is the first
Green obstruction \cite{Green}; for its $\acute{e}$tale descent to smooth
Deligne--Mumford superstacks, see
\cite[Proposition~2.2]{CorreaNojaCompactified}.  The class is
\begin{equation}\label{eq:green-class}
 \omega_2(S) \in H^1\!\left(M,T_M\otimes\bigwedge^2E^\vee\right).
\end{equation}
If $\omega_2(S)\neq0$, then $S$ is not projected.  Dually,
the class $\omega_2(S)$ in 
\eqref{eq:green-class} is represented by an extension
\begin{equation}\label{eq:abstract-extension}
 0\longrightarrow\bigwedge^2E^\vee
 \longrightarrow\mathscr X
 \longrightarrow\Omega_M^1
 \longrightarrow0.
\end{equation}

The reduced stack of $\M_g$ is the unrigidified spin moduli stack
$\SM_g=\SM_g^+\sqcup\SM_g^-$.  Its spin data are triples $(X,\eta,\varphi)$,
where $X$ is a smooth genus-$g$ curve and
$\varphi:\eta^{\otimes2}\xrightarrow{\sim}K_X$ is a chosen isomorphism.
We usually suppress
$\varphi$ from the notation.  The parity is $h^0(X,\eta)\bmod2$.  For a family of spin curves of genus $g\geq2$
$$
 \pi:(X,\eta)\longrightarrow T
$$
write
\begin{equation}\label{eq:F-G}
 F=\pi_*\eta^3,
 \qquad
 G=\pi_*K_{X/T}^2.
\end{equation}
For the moduli map $\kappa:T\to\SM_g$, $F$ is the pullback of the
odd cotangent bundle of $\M_g$ restricted to its reduced spin stack,
whereas $G=\kappa^*\Omega^1_{\SM_g}$ is the pullback of the ordinary
cotangent bundle of that stack.  Their ranks are $2g-2$ and $3g-3$, respectively
\cite[\S3.2 and Proposition~3.1]{DonagiWittenAtiyah}.

Let $p:X\times_T X\to T$ be the projection and let
$i_{\Delta_X}:X\hookrightarrow X\times_T X$ be the relative diagonal
embedding, with image $\Delta_X$.  Restriction to $\Delta_X$ gives
\begin{equation}\label{eq:diagonal-sequence}
0\longrightarrow \eta^3\boxtimes\eta^3
\longrightarrow (\eta^3\boxtimes\eta^3)(\Delta_X)
\longrightarrow (i_{\Delta_X})_*K_{X/T}^2\longrightarrow0,
\end{equation}
because $\cO(\Delta_X)|_{\Delta_X}\simeq K_{X/T}^{-1}$.
For $g\geq2$, Serre duality gives
$H^1(X_t,\eta_t^3)=H^0(X_t,\eta_t^{-1})^*=0$.
Relative K\"unneth therefore gives $p_*(\eta^3\boxtimes\eta^3)=F\otimes F$
and $R^1p_*(\eta^3\boxtimes\eta^3)=0$.
Taking direct images and then the even part under factor exchange yields
\begin{equation}\label{eq:DW}
0\longrightarrow\bigwedge^2F
\longrightarrow W_{\mathrm{DW}}
\longrightarrow G\longrightarrow0.
\end{equation}
We use the half-differential convention of
\cite[\S3.1, equation~(3.6)]{DonagiWittenAtiyah}: factor exchange acts by
minus ordinary transposition on $\eta^3\boxtimes\eta^3$.  Thus its even
regular tensors form $\bigwedge^2F$; the additional sign of the polar
denominator makes the quotient $G$ even.  We identify the exterior square
with these tensors by the unscaled antisymmetrization
$a\wedge b\mapsto a\otimes b-b\otimes a$.

By \cite[Proposition~3.1 and Theorem~3.2]{DonagiWittenAtiyah},
\eqref{eq:DW} represents
$$
 \kappa^*\omega_2(\M_g)
 \in H^1\!\left(T,G^\vee\otimes\bigwedge^2F\right),
$$
in the cotangent-extension formulation \eqref{eq:abstract-extension}.  We shall use two immediate consequences.  First, a
non-split pushout of \eqref{eq:DW} proves that \eqref{eq:DW} is non-split.
Second, if $h:T\to B$ is finite $\acute{e}$tale of degree $d>0$, then for
every vector bundle $V$ on $B$,
\begin{equation}\label{eq:trace-injective}
 \operatorname{Tr}_h\circ h^*=d\,\id
 \quad\text{on }H^1(B,V),
\end{equation}
so finite $\acute{e}$tale pullback preserves non-vanishing.

\section{The test family} \label{BD}

We now describe the complete spin test curve on which we evaluate
\eqref{eq:DW}.  Besides constructing the family, we prove the divisibility
and local descent statements needed later.

Fix a primitive cube root of unity $\zeta$.  Let $C$ be the smooth
projective model of the displayed affine curve, and let $\sigma$ be the
automorphism given on this chart by
\begin{equation}\label{eq:C-sigma}
 C:\ y^2=x^6-1,
 \qquad
 \sigma(x,y)=(\zeta x,-y).
\end{equation}
Then $C$ has genus two, while $\sigma$ has order six and no fixed points.
Indeed, at a finite fixed point one would have $x=y=0$, which does not lie
on $C$.  At infinity the two points are distinguished by
$v:=y/x^3=\pm1$, and $\sigma$ sends $v$ to $-v$.

\begin{remark}\label{rem:BD-correction}
This curve is identical to the final construction in
\cite{BryanDonagi}, with a typo corrected. In the displayed example there, 
both the curve and the order 6 automorphism are as in our \eqref{eq:C-sigma}, 
but there $\zeta$ is said to be a primitive sixth root of unity. That is a typo; 
$\zeta$ should be a primitive {\emph {third}} root of unity,
both there and here.
Otherwise, $\sigma$ would have fixed points at infinity.
To see this, we choose the coordinates at infinity to be
\[
u=1/x, v=y/(x^3).
\]
The equation of our curve becomes
\begin{equation} \label{eqn_at_infty}
v^2 = 1-u^6,
\end{equation}
the two points at infinity are
\[
u=0, v=\pm 1,
\]
and the action is
\[
u \mapsto \zeta^{-1} u, \ \ \
v \mapsto - \zeta^{-3} v.
\]
Thus with $\zeta$ a primitive sixth root of unity, each of these points is fixed, while with a primitive third root of unity they are exchanged. In fact, equation \eqref{eqn_at_infty} suggests a more symmetric way to visualize our curve $C$: it is the double cover of $\PP^1$ branched at the six roots of unity $r_i=\xi^i$, $i=1,\ldots,6$, in the $u$-coordinate, where
$\xi$ is the primitive sixth root with $\xi^2=\zeta$; subscripts are
understood modulo six. The automorphism $\sigma$ sends $r_i \mapsto r_{i-2}$ and interchanges the two sheets. (The wrong $\sigma$, based on $\zeta$ a primitive sixth root of unity, sends $r_i \mapsto r_{i-1}$ but does not exchange the sheets.)

We re-establish below the disjointness and
divisibility required for the resulting double-cover construction.
Note that the cyclic group $\langle\sigma\rangle$ does \emph{not} act freely on
$C$: some nontrivial powers of $\sigma$ have fixed points.  
We do not need the action to be free, as we never form the quotient by
$\langle\sigma\rangle$; the construction uses only the two graphs
$\Gamma_\rho$ and $\Gamma_{\sigma\rho}$.  Their disjointness requires
exactly that $\sigma$ itself have no fixed points.
\end{remark}

Let
$$
 \rho:\widetilde C\longrightarrow C
$$
be the connected $\acute{e}$tale cover corresponding to the kernel of
$\pi_1(C)\twoheadrightarrow H_1(C,\ZZ/2)$.  It has degree $16$ and genus
$17$.  In $\widetilde C\times C$ put
$$
 \Gamma_1=\Gamma_\rho,
 \qquad
 \Gamma_2=\Gamma_{\sigma\rho}.
$$
They are disjoint because $\sigma$ is fixed-point-free.

\begin{lemma}[Divisibility of the branch divisor]
\label{lem:branch-divisibility}
The line bundle $\cO(\Gamma_1+\Gamma_2)$ has a square root.
\end{lemma}

\begin{proof}
We first show that $[\Gamma_1+\Gamma_2]$ vanishes in
$H^2(\widetilde C\times C,\FF_2)$.  Its two pure K\"unneth components vanish:
the section component occurs twice, and the degree component is
$16+16$.  The mixed component of the graph of a map is induced by its map
on first cohomology.  But
$$
 \rho^*:H^1(C,\FF_2)\longrightarrow H^1(\widetilde C,\FF_2)
$$
is zero.  Indeed, every loop in $\widetilde C$ projects to an element of
the kernel of $\pi_1(C)\to H_1(C,\FF_2)$.  Hence the mixed component of
each $\Gamma_i$ already vanishes modulo two.

The integral group $H^2(\widetilde C\times C,\ZZ)$ is torsion-free, so
$c_1\cO(\Gamma_1+\Gamma_2)$ is twice an integral $(1,1)$-class.  By the
Lefschetz $(1,1)$ theorem, choose a line bundle with this half Chern class.
The remaining discrepancy lies in $\Pic^0(\widetilde C\times C)$, on which
multiplication by two is surjective.  Correcting by a square root in
$\Pic^0$ proves the assertion.
\end{proof}

Choose
\begin{equation}\label{eq:N}
 N^2\simeq\cO(\Gamma_1+\Gamma_2)
\end{equation}
and let
\begin{equation}\label{eq:double-cover}
 q:X\longrightarrow\widetilde C\times C
\end{equation}
be the associated double cover, branched along
$\Gamma_1+\Gamma_2$.  Explicitly,
$X=\operatorname{Spec}_{\widetilde C\times C}(\cO\oplus N^{-1})$,
where multiplication $N^{-2}\to\cO$ is given by the section of $N^2$
vanishing on $\Gamma_1+\Gamma_2$.  Along either branch section the local
equation is $t^2=z$, with $z$ a relative coordinate, so projection to
$\widetilde C$ is smooth.  Its fibre over $b\in\widetilde C$ is a double
cover of $C$ branched at $p=\rho(b)$ and $p'=\sigma\rho(b)$.
The nonempty reduced branch divisor makes every fibre connected, and
Riemann--Hurwitz gives genus four.

Set
\begin{equation}\label{eq:L}
 L=N(-\Gamma_2),
 \qquad
 L^2\simeq\cO(\Gamma_1-\Gamma_2).
\end{equation}
For $b\in\widetilde C$, write $L_b=L|_{\{b\}\times C}$ and define
\begin{equation}\label{eq:ell-map}
 \ell:\widetilde C\longrightarrow\Pic^2(C),
 \qquad
 b\longmapsto[K_C\otimes L_b].
\end{equation}
The multiplication-by-two map
$[2]:\Pic^1(C)\to\Pic^2(C)$ is finite $\acute{e}$tale of degree $16$.
Define
\begin{equation}\label{eq:T-prime}
 T':=\Pic^1(C)\mathop{\times}_{\Pic^2(C)}\widetilde C,
\end{equation}
where the two maps to $\Pic^2(C)$ are $[2]$ and $\ell$.  Denote the
projections by $a:T'\to\Pic^1(C)$ and $f':T'\to\widetilde C$.  Thus a
geometric point of $T'$ is a pair $(b,[A_b])$ with
$A_b^2\simeq K_C\otimes L_b$, and $f'$ is finite $\acute{e}$tale.

Fix a point $c_0\in C$.  By a normalized Poincar\'e line bundle on
$\Pic^1(C)\times C$ we mean a Poincar\'e line bundle $\mathscr P$
satisfying
$$
 \mathscr P|_{\{[M]\}\times C}\simeq M
 \quad\text{for every }[M]\in\Pic^1(C),
 \qquad
 \mathscr P|_{\Pic^1(C)\times\{c_0\}}
 \simeq\cO_{\Pic^1(C)}.
$$
The normalization removes the usual ambiguity of tensoring $\mathscr P$
by a line bundle pulled back from $\Pic^1(C)$.  Choose such a bundle and
put
$$
 A'=(a\times\id_C)^*\mathscr P,
 \qquad
 L'=(f'\times\id_C)^*L.
$$
If $t=(b,[A_b])$ is a geometric point of $T'$, then
$$
 A'|_{\{t\}\times C}\simeq A_b,
 \qquad
 L'|_{\{t\}\times C}\simeq L_b.
$$
Thus $A'$ remembers the chosen lift $[A_b]$, whereas $L'$ depends
only on $b \in \widetilde C$.  By the defining fibre-product relation for $T'$,
$$
 A_b^2\simeq K_C\otimes L_b.
$$
Equivalently, the different lifts over a fixed $b$ differ by a
two-torsion line bundle, which disappears after squaring.  Consequently,
$$
 (A')^2\otimes(K_C\otimes L')^{-1}
$$
is trivial on every fibre of $T'\times C\to T'$.  The see-saw principle
therefore gives a line bundle $\lambda$ on $T'$ such that
\begin{equation}\label{eq:base-twist}
 (A')^2\simeq
 K_C\otimes L'\otimes\operatorname{pr}_{T'}^*\lambda.
\end{equation}
The line bundle $\lambda$ records the failure of the fibrewise squaring
isomorphisms to glue to an isomorphism over $T'\times C$.

After replacing $T'$ by a connected component, choose a connected
$\acute{e}$tale double cover $\tau:T\to T'$.  Such a cover exists because
$T'$ is finite $\acute{e}$tale over the positive-genus curve
$\widetilde C$.  Since $\deg\tau=2$, the line bundle $\tau^*\lambda^{-1}$ has
even degree.  Surjectivity of multiplication by two on $\Pic^0(T)$
then gives a square root $B$ on $T$.  If $\varpi:T\times C\to T$ is the projection, set
$$
 f=f'\circ\tau,
 \qquad
 A=(\tau\times\id_C)^* A'\otimes\varpi^*B.
$$
Then $f:T\to\widetilde C$ is finite $\acute{e}$tale and
\begin{equation}\label{eq:A}
 A^2\simeq K_C\otimes L
\end{equation}
on $T\times C$.  Here and below we use the same symbol $L$ for both the original line bundle on
$\widetilde C\times C$
and for its
pullback to
$T\times C$.

The extra double cover $T\to T'$ serves only to remove the base twist in
\eqref{eq:base-twist} and to realize the square root as an actual line
bundle with a chosen squaring isomorphism.  The isomorphism class of the
spin fibre constructed below depends only on the image in $T'$.  Thus the
induced map on geometric moduli points factors through $T'$; no
injectivity of $T\to\SM_4^+$, or of $T'\to\SM_4^+$, is claimed or used.

Put $C_T=T\times C$.  The maps
$t\mapsto\rho(f(t))$ and $t\mapsto\sigma\rho(f(t))$ define two disjoint
sections of $C_T\to T$; we denote both the sections and their images by
$p$ and $p'$, respectively.  Let
$X_T=X\times_{\widetilde C}T$ and continue to write
$q:X_T\to C_T$ for the pulled-back double cover.  Let $r_2\subset X_T$ be
the ramification section over $p'$ and define
\begin{equation}\label{eq:eta}
 \eta=q^*A(r_2).
\end{equation}
Since $N=L(p')$ and $2r_2=q^*p'$, the double-cover canonical
bundle formula shows that $\eta$ is a theta characteristic:
\begin{equation}\label{eq:eta-spin}
 \eta^2=q^*(K_C N)=K_{X_T/T}.
\end{equation}

\begin{lemma}\label{lem:parity-and-lift}
The family $(X_T,\eta)\to T$ has even parity.  Moreover, the deck involution
$\iota$ has a lift $J$ to $\eta$ satisfying $J^2=-1$ and preserving
the chosen spin isomorphism $\eta^2\simeq K_{X_T/T}$.
\end{lemma}

\begin{proof}
Fix $t\in T$ and write
$q_t:X_t\to C$, $r_{2,t}=r_2|_{X_t}$,
$A_t=A|_{\{t\}\times C}$, $L_t=L|_{\{t\}\times C}$, and
$N_t=N|_{\{t\}\times C}$.
The invariant and anti-invariant parts of the fibrewise pushforward are
\begin{equation}\label{eq:fibre-pushforward-r2}
 (q_t)_*\cO_{X_t}(r_{2,t})
 \simeq\cO_C\oplus L_t^{-1}
 \qquad\text{as sheaves on }C.
\end{equation}
Indeed, away from the branch points this is the usual eigenspace
decomposition for a double cover, while at $p'_t$ the pole allowed by
$r_{2,t}$ changes the anti-invariant summand from $N_t^{-1}$ to
$N_t^{-1}(p'_t)=L_t^{-1}$.
By the projection formula and \eqref{eq:A},
\begin{align*}
 H^0(X_t,\eta_t)
 &\simeq H^0(C,A_t)\oplus H^0(C,A_tL_t^{-1})\\
 &\simeq H^0(C,A_t)\oplus H^0(C,K_C A_t^{-1}).
\end{align*}
Since $\deg A_t=1$ on the genus-two curve $C$, Riemann--Roch and Serre
duality give equal dimensions for the two summands.  Hence
$h^0(X_t,\eta_t)$ is even for every $t\in T$.

The factors $q^*A$ and $\cO_{X_T}(r_2)$ come with natural lifts of the action 
of $\iota$, giving also a lift $J_0$ of $\iota$ to $\eta$ with $J_0^2=1$.
Under the chosen spin isomorphism, compare $J_0\otimes J_0$
with the natural action of $\iota$ on $K_{X_T/T}$.  Their ratio is an
invertible function on $X_T$, hence constant on each proper connected
fibre.  At either ramification section, $J_0\otimes J_0$ acts by $+1$,
whereas the action on $K_{X_T/T}$ sends $dt$ to $-dt$ and acts by $-1$.
The ratio is therefore $-1$ on every fibre.  Therefore either lift
$\pm iJ_0$ preserves the spin isomorphism.  We choose
$$
 J:=-iJ_0,
 \qquad J^2=-1,
$$
to fix the eigenvalue convention used below.
\end{proof}

The chosen spin isomorphism and Lemma~\ref{lem:parity-and-lift}
define the moduli morphism $\kappa:T\to\SM_4^+$, where $T$ is a smooth
connected projective curve.  For the cubic spin bundle, the ordinary
deck-invariant and anti-invariant summands are obtained from
$$
 q_*\cO_{X_T}(3r_2)
 =\cO_{C_T}(p')\oplus N^{-1}(2p')
 =\cO_{C_T}(p')\oplus L^{-1}(p').
$$
Indeed, allowing a pole of order three in the ramification coordinate
gives the local generators $t^{-2}$ and $t^{-3}$.
The projection formula and \eqref{eq:eta} now give
\begin{equation}\label{eq:E-F}
 q_*\eta^3=\cE\oplus\cF,
 \qquad
 \cE=A^3(p'),
 \qquad
 \cF=A^3L^{-1}(p').
\end{equation}
The displayed $\cE$ and $\cF$ are respectively invariant and
anti-invariant under $J_0$.  Since $J=-iJ_0$ acts on $\eta^3$ by
$(-i)^3J_0^{\otimes3}=iJ_0^{\otimes3}$, $\cE$ is the $+i$-eigenbundle
and $\cF$ the $-i$-eigenbundle.  They have degree four on
each genus-two fibre, and
\begin{equation}\label{eq:EF-ratio}
 \cE\cF=K_C^3(p+p'),
 \qquad
 \cE\cF^{-1}=L.
\end{equation}
Recall that $\varpi:T\times C\to T$ is the projection, and write
$$
 F_i=\varpi_*\cE,
 \qquad
 F_{-i}=\varpi_*\cF,
 \qquad
 V=F_i\otimes F_{-i}.
$$
Since degree four is greater than $2g(C)-2=2$, the fibrewise first
cohomology of $\cE$ and $\cF$ vanishes.  Cohomology and base change
therefore show that $F_i$ and $F_{-i}$ are vector bundles of rank $3$ on
$T$; in particular, $V$ is a vector bundle of rank $9$ on $T$.

\begin{lemma}[Descent of the invariant cross block]
\label{lem:cross-block}
With the above ordering of the $J$-eigenbundles and the residue
normalization specified in the proof, the $J$-invariant part of the
extension \eqref{eq:DW} is identified,
as an extension of vector bundles on $T$, 
with
\begin{equation}\label{eq:cross-block}
0\longrightarrow V
\longrightarrow
p_{12*}\!\left((\cE\boxtimes\cF)(\Delta_C)\right)
\longrightarrow G_T\longrightarrow0,
\end{equation}
where
$p_{12}:C_T\times_T C_T=T\times C\times C\to T$ is the projection,
$\Delta_C\subset C_T\times_T C_T$ is the relative diagonal, and
$\cE\boxtimes\cF$ abbreviates
$\operatorname{pr}_1^*\cE\otimes\operatorname{pr}_2^*\cF$.  The quotient
is the rank-five vector bundle on $T$
\begin{equation}\label{eq:GT}
 G_T=\varpi_*K_C^2(p+p').
\end{equation}
\end{lemma}

\begin{proof}
We first identify the kernel and quotient globally.  The lift $J$ acts on
$F_i$ and $F_{-i}$ with eigenvalues $i$ and $-i$, respectively.  Hence
$$
 \left(\bigwedge^2(F_i\oplus F_{-i})\right)^J
 =F_i\otimes F_{-i}=V,
$$
because $J$ acts by $-1$ on $\bigwedge^2F_i$ and
$\bigwedge^2F_{-i}$ and by $+1$ on the cross term.  Moreover,
$$
 K_{X_T/T}^2=q^*K_C^2(p+p'),
$$
and the natural invariant summand of $q_*K_{X_T/T}^2$ is
$K_C^2(p+p')$.  Locally, $z=t^2$ identifies an invariant quadratic
differential $a(z)(\dd t)^2$ with
$a(z)(\dd z)^2/(4z)$, so precisely a simple pole is allowed at either
branch point.  Pushing forward by $\varpi$ gives $G_T$.

It remains to identify the invariant polar middle term globally.
Put $Y=C_T\times_T C_T$ and
$$
 H_{\cE\cF}=(\cE\boxtimes\cF)(\Delta_C),
 \qquad
 H_{\cF\cE}=(\cF\boxtimes\cE)(\Delta_C).
$$
Write
$\mathscr A=(\eta^3\boxtimes\eta^3)(\Delta_{X_T})$ for the upstairs
polar sheaf.  Before imposing exchange symmetry, its simultaneous
$J$-invariant pushforward $((q\times_Tq)_*\mathscr A)^J$ is the
residue-matching sheaf
$$
 \mathscr P_{\mathrm{res}}=
 \ker\!\left(
 H_{\cE\cF}\oplus H_{\cF\cE}
 \xrightarrow{\operatorname{res}_1-\operatorname{res}_2}
 (i_{\Delta_C})_*(\cE\cF K_C^{-1})
 \right).
$$
Here $i_{\Delta_C}:C_T\hookrightarrow Y$ is the diagonal embedding;
the two residue targets are identified by ordinary tensor transposition,
and both residues use the same convention.  Away from the branch
sections, this follows by separating the two sheets: absence of a pole
along the graph of the deck involution is exactly residue matching.
We check the full local module at both branch sections.

Work over a sufficiently small open subset of $C_T$ on which the cover
is $z=t^2$.  On the two factors write $z=t^2$ and $w=s^2$.
The inverse image of $z=w$ consists of $t=s$ and $t=-s$, and
\begin{equation}\label{eq:local-pole}
 \frac1{t-s}=\frac{t+s}{z-w}.
\end{equation}
Near the section over $p$, choose a local frame of $A^3$, where
$A$ is the line bundle in \eqref{eq:A}, and suppress its pullbacks on
the two factors.  Since the preimage of this neighbourhood is disjoint from $r_2$,
\eqref{eq:eta} identifies these with local frames of $\eta^3$.
The generators of the invariant and anti-invariant summands are $1$ and
$t$.  Every simultaneous $J$-invariant polar section has the form
$$
 \frac{a(z,w)+ts\,b(z,w)}{t-s}
 =\frac{s(a+zb)+t(a+wb)}{z-w},
$$
with $a,b$ regular: the spin-factor sign and the denominator sign cancel,
so the numerator is even under $(t,s)\mapsto(-t,-s)$.
The coefficients of the $\cE\boxtimes\cF$ and
$\cF\boxtimes\cE$ orientations are $U=a+zb$ and $V=a+wb$.
They satisfy $U-V=(z-w)b$, which is precisely equality of residues.
Conversely, a residue-matching pair reconstructs a regular numerator by
$$
 b=\frac{U-V}{z-w},\qquad a=U-zb.
$$
Thus the calculation proves both inclusion and surjectivity.

Near $r_2$, the generators of $q_*\cO_{X_T}(3r_2)$ are $t^{-2}$ and
$t^{-3}$, and the basic identity is
\begin{equation}\label{eq:local-pole-r2}
 \frac{t^{-3}s^{-3}}{t-s}
 =\frac{t^{-2}s^{-3}+t^{-3}s^{-2}}{z-w}.
\end{equation}
For a general invariant numerator the corresponding formula is
$$
 \frac{(a+tsb)t^{-3}s^{-3}}{t-s}
 =\frac{(a+wb)t^{-2}s^{-3}+(a+zb)t^{-3}s^{-2}}{z-w}.
$$
Again the two coefficients have the same residue.  Conversely, writing
them as $U'=a+wb$, $V'=a+zb$, residue matching gives the regular functions
$b=(V'-U')/(z-w)$ and $a=U'-wb$.
There is no additional branch-supported twist.  These descriptions
agree with the character decomposition off the diagonal and with the
unramified calculation, hence glue to the stated identification with
$\mathscr P_{\mathrm{res}}$.

Factor exchange moves the base point of $Y$, so its even sections do not
form an $\cO_Y$-submodule.  After applying $p_{12*}$, exchange is
$\cO_T$-linear.  In the convention of Section~\ref{obstruction}, its
even pairs, as local sections over $T$, are exactly
$$
 (\alpha,-\alpha^{\mathrm t}),
 \qquad \alpha\in p_{12*}H_{\cE\cF},
$$
where $\alpha^{\mathrm t}$ denotes ordinary transposition, including
interchange of the two curve coordinates.  Since this sends $z-w$ to
$-(z-w)$,
$$
 \operatorname{res}_{\Delta_C}(-\alpha^{\mathrm t})
 =\operatorname{res}_{\Delta_C}(\alpha).
$$
Thus every such pair satisfies residue matching, and projection to the
first orientation gives an isomorphism of bundles on $T$,
$$
 \bigl(p_{12*}\mathscr P_{\mathrm{res}}\bigr)^+
 \xrightarrow{\sim}p_{12*}H_{\cE\cF}.
$$

We finally specify the scalar in this comparison of extensions.
With the unscaled antisymmetrization fixed in Section~\ref{obstruction},
first-orientation projection is the identity on the kernel $V$.
The two polar residues add upstairs: under the natural identification of
the invariant quadratic differentials with $K_C^2(p+p')$, the upstairs
residue of $(\alpha,-\alpha^{\mathrm t})$ is
$2\operatorname{res}_{\Delta_C}(\alpha)$.
To check the scalar, at either branch point take $z=t^2$ and a spin frame
$e$ with $\varphi(e^2)=2\dd t$.  The polar section
$e(t)^3e(s)^3/(t-s)$ has upstairs residue $8(\dd t)^2$.
Either selected orientation has downstairs residue $(\dd z)^2/z$,
whose pullback is $4(\dd t)^2$; at $p'$ the two regular eigengenerators
are interchanged, with the same result.
Accordingly the comparison with \eqref{eq:cross-block}, whose quotient
map is the usual residue of the selected orientation, uses
$\tfrac12\id_{G_T}$ on the quotient.  This fixes the normalization and
gives an isomorphism of extensions.  Equivalently, with the natural
outer-term identifications, the class of \eqref{eq:cross-block} is twice
the class of the $J$-invariant extension.
\end{proof}

\section{Reduction to a mixed Atiyah class on the fixed curve}\label{Atiyah}

 We have isolated a concrete summand of the universal obstruction.  We now
restrict its middle sheaf from $C_T\times_T C_T$ to the first
infinitesimal neighbourhood $2\Delta_C$ (the doubled diagonal).  After direct image, this restriction realizes
the categorical pushout of the cross-block extension along the multiplication
map.  We then identify the resulting extension class with the direct image of
a mixed Atiyah class and descend the latter from $T$ to the fixed curve $C$.

We first fix the Atiyah-class notation used in this section.  For a line
bundle $\mathscr M$ on $C_T=T\times C$, its \emph{relative Atiyah class}
is the class of the first relative jet sequence
\begin{equation}\label{eq:relative-jet-sequence}
0\longrightarrow
 \mathscr M\otimes\Omega^1_{C_T/T}
\longrightarrow J^1_{C_T/T}(\mathscr M)
\longrightarrow\mathscr M\longrightarrow0.
\end{equation}
Thus
\begin{equation}\label{eq:relative-At-target}
 \At_{{C_T}/T}(\mathscr M)
 \in\operatorname{Ext}^1_{C_T}
   (\mathscr M,\mathscr M\otimes\Omega^1_{C_T/T})
 =H^1(C_T,\Omega^1_{C_T/T}).
\end{equation}
Since $C_T$ is a product and
$\Omega^1_{C_T/T}=\operatorname{pr}_C^*K_C$, K\"unneth gives
\begin{equation}\label{eq:At-Kunneth-T}
 H^1(C_T,\Omega^1_{C_T/T})
 \simeq
 H^1(T,\cO_T)\otimes H^0(C,K_C)
 \oplus
 H^0(T,\cO_T)\otimes H^1(C,K_C).
\end{equation}
We call the first summand the \emph{mixed component} and the second the
\emph{vertical component}.  The restriction of the vertical component to
a fibre $\{t\}\times C$ is the first Chern class of $\mathscr M_t$ and is
therefore determined by $\deg\mathscr M_t$.  In particular, a line bundle
of relative degree zero has no vertical component.  This applies below to
$L=\cE\cF^{-1}$.

Multiplication of sections defines
\begin{equation}\label{eq:mu}
 \mu:V\longrightarrow
 R_T:=\varpi_*(\cE\cF)=\varpi_*(K_C^3(p+p')).
\end{equation}
This is a morphism from the rank-nine bundle $V$ to the rank-seven bundle
$R_T$, both on $T$.

Let $2\Delta_C\subset C_T\times_T C_T$ be the closed subscheme defined by
$\mathcal I_{\Delta_C}^2$, and put
$$
 \cH=(\cE\boxtimes\cF)(\Delta_C).
$$
Thus $2\Delta_C$ is the first infinitesimal neighbourhood of the relative
diagonal, and
$$
 \cH|_{2\Delta_C}
 =\cH\otimes\cO_{2\Delta_C}
$$
is the restriction of $\cH$ to that neighbourhood.

To regard it as a bundle on $C_T\simeq\Delta_C$, use the symmetric
retraction
$$
 r:2\Delta_C\longrightarrow\Delta_C\simeq C_T,
 \qquad
 r^\#(f)=\frac{\operatorname{pr}_1^*f+\operatorname{pr}_2^*f}{2}
 \pmod{\mathcal I_{\Delta_C}^2}.
$$
Note that $r^\#$ is a ring map: its multiplicativity defect is
$$
 \frac14(\operatorname{pr}_1^*f-\operatorname{pr}_2^*f)
         (\operatorname{pr}_1^*g-\operatorname{pr}_2^*g)=0
 \quad\text{on }2\Delta_C.
$$
We can therefore put $\mathscr Q=r_*(\cH|_{2\Delta_C})$.
The map $r$ is a retraction invariant under exchange, and
$p_{12}|_{2\Delta_C}=\varpi\circ r$.
The sheaf $\mathscr Q$ is a rank-two $\cO_{C_T}$-bundle.

\begin{proposition}[Restriction and categorical pushout]
\label{prop:second-diagonal}
The restriction morphism
$$
 p_{12*}\cH
 \longrightarrow
 p_{12*}\bigl(\cH|_{2\Delta_C}\bigr)
$$
induces the multiplication map $\mu:V\to R_T$ on the kernels of the
corresponding extensions and the identity on their common quotient $G_T$.
The resulting exact sequence
\begin{equation}\label{eq:jet-extension}
0\longrightarrow R_T
\longrightarrow
p_{12*}\!\left(
  \cH|_{2\Delta_C}
\right)
\longrightarrow G_T\longrightarrow0
\end{equation}
is therefore canonically isomorphic to the categorical pushout of
\eqref{eq:cross-block} along $\mu$.

Its extension class is obtained by applying $\varpi_*$ to the
extension \eqref{eq:layers} below, with middle term $\mathscr Q$.
Under the canonical identification
$$
 \operatorname{Ext}^1_{C_T}(\cE\cF K_C^{-1},\cE\cF)
 =H^1(C_T,\operatorname{pr}_C^*K_C),
$$
the class of that extension is
\begin{equation}\label{eq:half-At}
 \frac12\At_{C_T/T}(\cE\cF^{-1})
 =\frac12\At_{C_T/T}(L).
\end{equation}
Thus, after direct image
by $\varpi$, 
the class of \eqref{eq:jet-extension} lies in
\begin{equation}\label{eq:jet-extension-target}
 \operatorname{Ext}^1_T(G_T,R_T)
 =H^1\!\left(T,\Hom(G_T,R_T)\right).
\end{equation}
The equality is understood up to the global residue-sign convention.
\end{proposition} 

\begin{proof}
Let $\mathcal I_{\Delta_C}$ be the ideal sheaf of the diagonal.  The
conormal filtration of $\cH|_{2\Delta_C}$ gives
$$
0\longrightarrow
 \left(\mathcal I_{\Delta_C}/\mathcal I_{\Delta_C}^2\right)
 \otimes\cH|_{\Delta_C}
\longrightarrow
 \cH|_{2\Delta_C}
\longrightarrow
 \cH|_{\Delta_C}
\longrightarrow0.
$$
Since
$$
 \cO(\Delta_C)|_{\Delta_C}\simeq K_C^{-1},
 \qquad
 \mathcal I_{\Delta_C}/\mathcal I_{\Delta_C}^2\simeq K_C,
$$
we have
$$
 \cH|_{\Delta_C}\simeq\cE\cF K_C^{-1},
 \qquad
 \left(\mathcal I_{\Delta_C}/\mathcal I_{\Delta_C}^2\right)
 \otimes\cH|_{\Delta_C}\simeq\cE\cF.
$$
Pushing the conormal filtration forward along the finite retraction
$r$, the two layers give the exact sequence of $\cO_{C_T}$-modules
\begin{equation}\label{eq:layers}
0\longrightarrow\cE\cF
\longrightarrow\mathscr Q=r_*(\cH|_{2\Delta_C})
\longrightarrow\cE\cF K_C^{-1}
\longrightarrow0.
\end{equation}

The outer terms are $K_C^3(p+p')$ and $K_C^2(p+p')$.  On every genus-two
fibre they have degrees $8$ and $6$, respectively, so their first
cohomology vanishes.  Cohomology and base change therefore show that
applying $\varpi_*$ to \eqref{eq:layers} gives an exact sequence of vector
bundles on $T$, namely \eqref{eq:jet-extension}, since
$\varpi_*\mathscr Q=p_{12*}(\cH|_{2\Delta_C})$.

The restriction morphism $\cH\to\cH|_{2\Delta_C}$ now gives the
commutative diagram
$$
\begin{tikzcd}[column sep=small]
0 \arrow[r] & V \arrow[r] \arrow[d,"\mu"]
  & p_{12*}\cH \arrow[r] \arrow[d,"\mathrm{res}"]
  & G_T \arrow[r] \arrow[d,equal] & 0\\
0 \arrow[r] & R_T \arrow[r]
  & p_{12*}(\cH|_{2\Delta_C}) \arrow[r]
  & G_T \arrow[r] & 0.
\end{tikzcd}
$$
The middle vertical arrow is induced by restriction to $2\Delta_C$.
On the kernel
$$
 V=\varpi_*\cE\otimes\varpi_*\cF
$$
it is fibrewise multiplication, hence it induces precisely the map
$\mu:V\to R_T$.

The universal property of the categorical pushout therefore produces a
morphism from the pushout of the top row along $\mu$ to the bottom row.
This morphism is the identity on both $R_T$ and $G_T$, and is consequently
an isomorphism.  Thus restriction of the middle sheaf to $2\Delta_C$,
followed by direct image, realizes the categorical pushout of
\eqref{eq:cross-block} along $\mu$.

We now identify its extension class.  Choose an open cover
$\{U_a\}$ of $C_T$ carrying relative coordinates and frames of $\cE$ and
$\cF$.  On $U_a\times_TU_a$, write
$z=(x+y)/2$ and $u=x-y$; the coordinate $z$ realizes the retraction
$r$, and the conormal class of $u$ is identified with $\dd z$.
Suppose that on $U_{ab}$ the two frames change
by relative transition functions $g_{ab}$ and $k_{ab}$.  Modulo $u^2$,
$$
 \frac{g_{ab}(z+u/2)k_{ab}(z-u/2)}
      {g_{ab}(z)k_{ab}(z)}
 =1+\frac{u}{2}
   \left(\partial_z\log g_{ab}
        -\partial_z\log k_{ab}\right).
$$
Consequently, after the leading polar term $1/u$ has been matched, the
constant term in the change of the local splitting of
\eqref{eq:layers} is
$$
 \frac12\left(
 \dd_{C_T/T}\log g_{ab}-\dd_{C_T/T}\log k_{ab}
 \right).
$$
These forms constitute the \v Cech representative of
$\frac12\At_{C_T/T}(\cE\cF^{-1})$ in
$H^1(C_T,\Omega^1_{C_T/T})$.  For a change of relative coordinate $v=\phi(x)$, the polar denominator
transforms as
$$
 \phi(z+u/2)-\phi(z-u/2)=\phi'(z)u+O(u^3).
$$
After matching the leading polar term, its reciprocal has no constant
correction.  Thus the coordinate change adds no further term to the
splitting cocycle.  This calculation identifies the global class of
\eqref{eq:layers}; the direct-image cocycle for \eqref{eq:jet-extension}
is made explicit in Proposition~\ref{prop:reduction} below.  Finally,
\eqref{eq:EF-ratio} gives \eqref{eq:half-At}.
\end{proof}

We now introduce the map to the fixed curve that is needed for the descent:
\begin{equation}\label{eq:h-map}
 h:=\rho\circ f:T\longrightarrow C.
\end{equation}
It is finite $\acute{e}$tale of positive degree: a nonempty connected
component of a finite $\acute{e}$tale cover of the connected curve
$\widetilde C$ has open and closed, hence surjective, image.  On $C\times C$, let $p_1,p_2$ be the two
projections, regard the first factor as the parameter, and let
$\Gamma_\sigma=\{(x,\sigma(x)):x\in C\}$.  Define the vector bundles on
the first copy of $C$
\begin{equation}\label{eq:En}
 E_n=p_{1*}\!\left(p_2^*K_C^n(\Delta+\Gamma_\sigma)\right),
 \qquad n=2,3,
\end{equation}
of ranks $5$ and $7$, respectively.  Since the inverse images of
$\Delta$ and $\Gamma_\sigma$ under $h\times\id_C$ are the divisors $p$
and $p'$, flat base change gives canonical isomorphisms
\begin{equation}\label{eq:GT-RT-pullback}
 E_{2,T}:=h^*E_2\simeq G_T,
 \qquad
 E_{3,T}:=h^*E_3\simeq R_T.
\end{equation}
This notation makes explicit the parallel roles of the two bundles on
$C$ and their pullbacks to $T$.

Multiplication on the second factor defines a morphism of vector bundles
on $C$
\begin{equation}\label{eq:j}
 j:H^0(C,K_C)\otimes\cO_C\longrightarrow\cHom(E_2,E_3),
 \qquad
 s\longmapsto m_s,
\end{equation}
where $m_s$ is multiplication by the canonical form $s$ on the moving
second copy of $C$.  Finally put
\begin{equation}\label{eq:Lcal}
 \cL=\cO_{C\times C}(\Delta-\Gamma_\sigma).
\end{equation}
Its relative Atiyah class, relative to $p_1$, belongs to
$H^1(C\times C,p_2^*K_C)$.  Since $\cL$ has degree zero on every fibre of
$p_1$, its vertical component vanishes.  We denote its mixed component by
\begin{equation}\label{eq:Atmix-space}
 \At_{\mathrm{mix}}(\cL)
 \in H^1(C,\cO_C)\otimes H^0(C,K_C).
\end{equation}

\begin{proposition}[Cohomological reduction]
\label{prop:reduction}
Under the identifications \eqref{eq:GT-RT-pullback}, the class of
\eqref{eq:jet-extension} is
\begin{equation}\label{eq:pullback-key-class}
 \pm\frac14\,
 h^*H^1(j)\bigl(\At_{\mathrm{mix}}(\cL)\bigr)
 \in H^1\!\left(T,\Hom(E_{2,T},E_{3,T})\right),
\end{equation}
where the sign depends only on the residue and \v Cech conventions.
Consequently, if
\begin{equation}\label{eq:key-nonzero}
 H^1(j)\bigl(\At_{\mathrm{mix}}(\cL)\bigr)\neq0,
\end{equation}
then $\omega_2(\M_4^+)\neq0$.
\end{proposition}

\begin{proof}
The pullback identities for the two outer bundles are
\eqref{eq:GT-RT-pullback}.  Multiplication by a relative canonical form
defines
$$
 j_T:H^0(C,K_C)\otimes\cO_T
 \longrightarrow\Hom(E_{2,T},E_{3,T}),
$$
and flat base change identifies this morphism with $h^*j$.  The morphism
$j_T$ acts on the direct-image bundles by multiplication by a relative
canonical form.  It is distinct from the kernel morphism
$\mu:V\to R_T$ along which the categorical pushout of
\eqref{eq:cross-block} is formed.

The divisor identity \eqref{eq:L} gives the equality of line bundles on
$C_T$
\begin{equation}\label{eq:L-square-pullback}
 L^2=(h\times\id_C)^*\cL.
\end{equation}
Naturality and additivity of the relative Atiyah class give
\begin{equation}\label{eq:At-pullback-identity}
 2\At_{C_T/T}(L)
 =(h\times\id_C)^*
   \At_{(C\times C)/C}(\cL).
\end{equation}
Both sides have zero vertical component.  Under the mixed summand of the
K\"unneth decomposition
$$
H^1(T\times C,\operatorname{pr}_C^*K_C)
\simeq H^1(T,\cO_T)\otimes H^0(C,K_C)
 \oplus H^0(T,\cO_T)\otimes H^1(C,K_C),
$$
the class $\At_{C_T/T}(L)$ corresponds to
\begin{equation}\label{eq:beta-class}
 \beta:=\frac12h^*\At_{\mathrm{mix}}(\cL)
 \in H^1(T,\cO_T)\otimes H^0(C,K_C)
 =H^1\!\left(T,H^0(C,K_C)\otimes\cO_T\right).
\end{equation}

To identify the direct image, choose an affine cover $\{U_a\}$ of
$T$ (or a Stein cover in the analytic setting).  On a whole inverse image
$U_a\times C$, K\"unneth gives
$$
 H^1(U_a\times C,\operatorname{pr}_C^*K_C)
 =\Gamma(U_a,\cO_T)\otimes H^1(C,K_C).
$$
By Proposition~\ref{prop:second-diagonal}, the class of
\eqref{eq:layers} is $\frac12\At_{C_T/T}(L)$ up to the fixed sign.
Since $L$ has degree zero on each fibre, this class has zero vertical
component, so its restriction to $U_a\times C$ vanishes.
Therefore \eqref{eq:layers} admits an $\cO_{U_a\times C}$-linear splitting
over each whole inverse image.

Differences of these splittings are global sections on $U_{ab}\times C$ of
$\cHom(\cE\cF K_C^{-1},\cE\cF)=\operatorname{pr}_C^*K_C$, hence are forms
$$
 \theta_{ab}\in\Gamma(U_{ab},\cO_T)\otimes H^0(C,K_C).
$$
Under the Leray--K\"unneth identification of the kernel of fibre
restriction with $H^1(T,\varpi_*\operatorname{pr}_C^*K_C)$, their cocycle
represents $\frac12\beta$, up to the same overall sign.
Applying $\varpi_*$ to these splittings gives genuine local splittings
of \eqref{eq:jet-extension} on $U_a$.  Their differences are multiplication
by $\theta_{ab}$, namely $j_T(\theta_{ab})$.
It follows that the class of \eqref{eq:jet-extension} is
$$
 \pm\frac12H^1(j_T)(\beta)
 =\pm\frac14h^*H^1(j)
   \bigl(\At_{\mathrm{mix}}(\cL)\bigr)
$$
as asserted in \eqref{eq:pullback-key-class}.

Finally, $h$ is finite $\acute{e}$tale, so
\eqref{eq:trace-injective} shows that a nonzero class in
\eqref{eq:key-nonzero} remains nonzero after pullback to $T$.  Hence
\eqref{eq:jet-extension} is non-split.  By
Proposition~\ref{prop:second-diagonal}, this extension is the categorical pushout of
the invariant cross block \eqref{eq:cross-block} along $\mu$.  Since every
pushout of a split extension is split, the invariant cross block is
non-split.  By the normalized comparison in Lemma~\ref{lem:cross-block},
the $J$-invariant part of \eqref{eq:DW} is therefore non-split.
Since $J$ fixes $T$ and has finite order, averaging gives
$\cO_T$-linear projections onto the invariant subbundles, commuting
with the maps of \eqref{eq:DW}.  Write
$P_W:W_{\mathrm{DW}}\to W_{\mathrm{DW}}^J$ for the middle projection.
A splitting $s:G\to W_{\mathrm{DW}}$ of \eqref{eq:DW} would give a
splitting $P_W\circ s|_{G^J}$ of the invariant sequence.
Hence the full extension \eqref{eq:DW} is non-split.

\end{proof}

\section{The non-vanishing calculation} \label{calculation}

It remains to prove \eqref{eq:key-nonzero}.  We now compute the relevant
elementary transforms and the mixed Atiyah class explicitly on
$C:y^2=x^6-1$.

Set
\begin{equation}\label{eq:u-basis}
 u_0=\frac{\dd x}{y},
 \qquad
 u_1=x\frac{\dd x}{y}.
\end{equation}
For $V_n=H^0(C,K_C^n)$ use the bases
\begin{align}
 V_1&=(u_0,u_1),\notag\\
 V_2&=(q_0,q_1,q_2)=(u_0^2,u_0u_1,u_1^2),\notag\\
 V_3&=(r_0,r_1,r_2,r_3,r_4)\notag\\
    &=(u_0^3,xu_0^3,x^2u_0^3,x^3u_0^3,yu_0^3).
\label{eq:V-bases}
\end{align}
The last generator is the extra hyperelliptic generator of the canonical
ring.  Let
$\rho_n=\sigma^*|_{V_n}\in\operatorname{GL}(V_n)$.  In the ordered bases
of \eqref{eq:V-bases}, these endomorphisms are
\begin{align}
 \rho_1&=\operatorname{diag}(-\zeta,-\zeta^2),\notag\\
 \rho_2&=\operatorname{diag}(\zeta^2,1,\zeta),\notag\\
 \rho_3&=\operatorname{diag}(-1,-\zeta,-\zeta^2,-1,1).
\label{eq:rho-matrices}
\end{align}
Put
\begin{equation}\label{eq:D}
 D=\rho_3^{-1}=\operatorname{diag}(-1,-\zeta^2,-\zeta,-1,1).
\end{equation}

Since $H^1(C,K_C^n)=0$ for $n=2,3$, the residue sequence along
the two disjoint graphs in \eqref{eq:En} gives the vector-bundle extension
\begin{equation}\label{eq:En-extension}
0\longrightarrow V_n\otimes\cO_C
\longrightarrow E_n
\longrightarrow Q_n\longrightarrow0,
\qquad
Q_n=K_C^{n-1}\oplus\sigma^*K_C^{n-1}.
\end{equation}
Thus $E_2$ and $E_3$ have ranks $5$ and $7$, while $Q_2$ and $Q_3$ have
rank $2$.  The extension class of \eqref{eq:En-extension} lives in
\begin{equation}\label{eq:en-class-target}
 e_n\in
 \operatorname{Ext}^1_C(Q_n,V_n\otimes\cO_C)
 =H^1\!\left(C,\cHom(Q_n,V_n\otimes\cO_C)\right).
\end{equation}
Identify the second summand with $K_C^{n-1}$ by the pullback differential.
Under Serre duality and our residue convention, the elementary-transform
class of the diagonal is the canonical tensor $I\in V_n\otimes V_n^*$.
This can be verified directly.  Put $B_n=K_C^{1-n}$.  Relative Serre
duality identifies the dual of the single-diagonal extension with the
direct-image sequence obtained from
$$
 0\longrightarrow p_2^*B_n(-\Delta)
 \longrightarrow p_2^*B_n
 \longrightarrow (i_\Delta)_*B_n\longrightarrow0.
$$
Its connecting map on the parameter curve is, up to the common residue
sign, diagonal restriction
$$
 H^1(C\times C,p_2^*B_n)\longrightarrow H^1(C,B_n).
$$
Since $H^0(C,B_n)=0$, K\"unneth identifies this map with the identity.
The residue normalization is chosen so that the resulting diagonal
tensor is $I$ for both $n=2$ and $n=3$.
The single-graph family for $\Gamma_\sigma$ is obtained from the diagonal
family by pullback by $\sigma$ on the parameter.  For
$c\in H^1(C,K_C^{1-n})$ and $s\in V_n$, Serre duality gives
$$
 \langle\sigma^*c,s\rangle
 =\langle c,(\sigma^{-1})^*s\rangle.
$$
Thus pullback on the cohomology factor is dual to $\rho_n^{-1}$ and
transforms the graph tensor into $\rho_n^{-1}$.
In this fixed convention, the class of \eqref{eq:En-extension} is
\begin{equation}\label{eq:en-class}
 e_n=(I,\rho_n^{-1}).
\end{equation}

For $q\in V_2$, let $N_q:V_1\to V_3$ be multiplication by $q$.  The
connecting map for \eqref{eq:En-extension} with $n=3$ is
\begin{equation}\label{eq:delta3}
 \delta_3:V_2\oplus V_2\longrightarrow\Hom(V_1,V_3),
 \qquad
 (q,q')\longmapsto N_q+DN_{q'}.
\end{equation}
Here we have used
$H^0(C,Q_3)\simeq V_2\oplus V_2$ and, by Serre duality,
$H^1(C,V_3\otimes\cO_C)\simeq\Hom(V_1,V_3)$.
Explicitly, if $q=a_0q_0+a_1q_1+a_2q_2$, then in the bases
\eqref{eq:V-bases}
\begin{equation}\label{eq:Nq-matrix}
 N_q=
 \begin{pmatrix}
  a_0&0\\
  a_1&a_0\\
  a_2&a_1\\
  0&a_2\\
  0&0
 \end{pmatrix}
 \in\Hom(V_1,V_3).
\end{equation}

\begin{lemma}\label{lem:delta-injective}
The map $\delta_3$ is injective.
\end{lemma}

\begin{proof}
Write $q=\sum a_kq_k$ and $q'=\sum b_kq_k$.  
The condition for $(q,q')$ to be in the kernel of $\delta_3$ is that
for each $k=0,1,2$, two
successive nonzero entries in \eqref{eq:Nq-matrix} give
$$
 a_k+d_kb_k=0,
 \qquad
 a_k+d_{k+1}b_k=0,
$$
where
$$
 (d_0,d_1,d_2,d_3)=(-1,-\zeta^2,-\zeta,-1).
$$
Consecutive $d_k$ are distinct, so $a_k=b_k=0$ for every $k$.
\end{proof}

Under Serre duality,
\begin{equation}\label{eq:End-identification}
H^1(C,\cO_C)\otimes H^0(C,K_C)\simeq\End(V_1).
\end{equation}
Under \eqref{eq:End-identification}, the mixed diagonal class is
the identity correspondence on $H^0(C,K_C)$: cup product with the
diagonal class followed by pushforward acts as the identity.
The global Atiyah sign is fixed so that this tensor is $I$.  Since
$\cO(\Gamma_\sigma)=(\sigma\times\id_C)^*\cO(\Delta)$,
the same first-factor pullback and Serre-duality calculation gives
$\rho_1^{-1}$ for the graph contribution.  We use one uniform overall
Atiyah-sign convention.  Thus
\begin{equation}\label{eq:A-matrix}
 A:=\At_{\mathrm{mix}}(\cL)
 =I-\rho_1^{-1}
 =\operatorname{diag}(-\zeta,-\zeta^2).
\end{equation}
In particular, $A$ is an endomorphism of
$V_1=H^0(C,K_C)$, and the last matrix in \eqref{eq:A-matrix} is written in
the ordered basis $(u_0,u_1)$ from \eqref{eq:u-basis}.

\begin{proposition}\label{prop:explicit-nonzero}
For the multiplication morphism \eqref{eq:j},
$$
 H^1(j)(A)\neq0
 \quad\text{in}\quad
 H^1\!\left(C,\cHom(E_2,E_3)\right).
$$
\end{proposition}

\begin{proof}
Choose a sufficiently fine \v Cech cover
$\mathfrak U=\{U_a\}$ of $C$ on which both extensions
\eqref{eq:En-extension} split.  Relative to local splittings
$$
 E_n|_{U_a}\simeq
 (V_n\otimes\cO_{U_a})\oplus Q_n|_{U_a},
$$
the comparison maps between these splittings are automorphisms of
the restrictions of the fixed bundles
$(V_n\otimes\cO_C)\oplus Q_n$ to $U_{ab}$.
Writing $\chi_{n,a}:(V_n\otimes\cO_{U_a})\oplus Q_n|_{U_a}
\xrightarrow{\sim}E_n|_{U_a}$ for these identifications, the comparison
$T_{n,ab}=\chi_{n,a}^{-1}\chi_{n,b}$ has matrix
\begin{equation}\label{eq:Tnab}
 T_{n,ab}=
 \begin{pmatrix}I_{V_n}&e_{n,ab}\\0&I_{Q_n}\end{pmatrix}.
\end{equation}
Here
$$
 e_{n,ab}\in
 \Gamma\!\left(U_{ab},
 \cHom(Q_n,V_n\otimes\cO_C)\right),
$$
and the \v Cech cocycle $(e_{n,ab})$ represents the extension class
$e_n$ in \eqref{eq:en-class-target}.

Choose a \v Cech cocycle
$\alpha=(\alpha_{ab})\in
Z^1(\mathfrak U,V_1\otimes\cO_C)$ representing the class $A$ under
\eqref{eq:End-identification}.  Applying the sheaf morphism $j$ gives
$$
 Z_A:=j(\alpha)\in
 Z^1\!\left(\mathfrak U,\cHom(E_2,E_3)\right).
$$
In the splittings \eqref{eq:Tnab}, each local homomorphism
$Z_{A,ab}:E_2|_{U_{ab}}\to E_3|_{U_{ab}}$ has the form
\begin{equation}\label{eq:ZA-blocks}
 Z_{A,ab}=
 \begin{pmatrix}
  Z_{A,ab}^{\mathrm{ul}}&Z_{A,ab}^{\mathrm{ur}}\\
  0&Z_{A,ab}^{\mathrm{lr}}
 \end{pmatrix}.
\end{equation}
The lower-left block is zero because multiplication preserves the constant
subbundles $V_2\otimes\cO_C\subset E_2$ and
$V_3\otimes\cO_C\subset E_3$.  More explicitly,
$$
\begin{aligned}
 Z_{A,ab}^{\mathrm{ul}}
 &\in\Gamma\!\left(
 U_{ab},\cHom(V_2,V_3)\otimes\cO_C\right),\\
 Z_{A,ab}^{\mathrm{lr}}
 &\in\Gamma\!\left(U_{ab},\cHom(Q_2,Q_3)\right).
\end{aligned}
$$

We must determine whether the equation
\begin{equation}\label{eq:solve-coboundary}
 Z_A=\delta\Phi
\end{equation}
has a solution
$\Phi\in C^0(\mathfrak U,\cHom(E_2,E_3))$; this is precisely the equation
whose solvability would say that $H^1(j)(A)=0$.  Write a putative local
zero-cochain as
\begin{equation}\label{eq:Phi-blocks}
 \Phi_a=
 \begin{pmatrix}
  \phi_a&\beta_a\\
  \gamma_a&\psi_a
 \end{pmatrix}.
\end{equation}
The four entries in \eqref{eq:Phi-blocks} are, respectively, local sections
of
$$
\begin{gathered}
 \cHom(V_2,V_3)\otimes\cO_C,\qquad
 \cHom(Q_2,V_3\otimes\cO_C),\\
 \cHom(V_2\otimes\cO_C,Q_3),\qquad
 \cHom(Q_2,Q_3).
\end{gathered}
$$
We use the coboundary convention
$$
 (\delta\Phi)_{ab}
 =T_{3,ab}\Phi_bT_{2,ab}^{-1}-\Phi_a.
$$
Multiplication of these block matrices gives
\begin{align}
 (\delta\Phi)_{ab}^{\mathrm{ll}}
 &=\gamma_b-\gamma_a,\label{eq:delta-ll}\\
 (\delta\Phi)_{ab}^{\mathrm{ul}}
 &=\phi_b-\phi_a+e_{3,ab}\gamma_b,\label{eq:delta-ul}\\
 (\delta\Phi)_{ab}^{\mathrm{lr}}
 &=\psi_b-\psi_a-\gamma_be_{2,ab}.\label{eq:delta-lr}
\end{align}
Since the lower-left block of $Z_A$ vanishes,
\eqref{eq:solve-coboundary} and \eqref{eq:delta-ll} force the
$\gamma_a$ to glue to a global morphism
$$
 \gamma:V_2\otimes\cO_C\longrightarrow Q_3.
$$
Taking cohomology in the upper-left and lower-right blocks yields the two
necessary identities
\begin{align}
 [Z_A^{\mathrm{ul}}]
 &=[e_3\gamma]
 &&\text{in }H^1\!\left(
 C,\cHom(V_2,V_3)\otimes\cO_C\right),
 \label{eq:block-identity-ul}\\
 [Z_A^{\mathrm{lr}}]
 &=-[\gamma e_2]
 &&\text{in }H^1\!\left(C,\cHom(Q_2,Q_3)\right).
 \label{eq:block-identity-lr}
\end{align}

We first solve the upper-left identity.  Using
$H^0(C,Q_3)\simeq V_2\oplus V_2$, write the global morphism $\gamma$ as
\begin{equation}\label{eq:gamma-C1C2}
 \gamma=(C_1,C_2):V_2\longrightarrow V_2\oplus V_2,
 \qquad C_1,C_2\in\End(V_2).
\end{equation}
Here $(C_1,C_2)$ denotes the induced map on global sections;
the bundle morphism $\gamma$ is obtained from it by evaluation into $Q_3$.
Under Serre duality, evaluation of
$[Z_A^{\mathrm{ul}}]$ on $q\in V_2$ is the map
$N_qA\in\Hom(V_1,V_3)$.  Evaluation of $[e_3\gamma]$ is the connecting
map \eqref{eq:delta3} applied to $(C_1q,C_2q)$.  Therefore
\eqref{eq:block-identity-ul} is equivalent to
\begin{equation}\label{eq:upper-equation}
 N_qA=N_{C_1q}+DN_{C_2q}
 \qquad\text{for every }q\in V_2.
\end{equation}
Equivariance of multiplication gives
$\rho_3N_q=N_{\rho_2q}\rho_1$, and therefore
$$
 DN_{\rho_2q}=N_q\rho_1^{-1}.
$$
Since $A=I-\rho_1^{-1}$ by \eqref{eq:A-matrix}, the pair
\begin{equation}\label{eq:C1C2}
 C_1=I_3,
 \qquad
 C_2=-\rho_2
\end{equation}
solves \eqref{eq:upper-equation}.  It is the unique solution because
$\delta_3$ is injective by Lemma~\ref{lem:delta-injective}.

It remains to test the lower-right identity
\eqref{eq:block-identity-lr}.  Using the differential of $\sigma$ to
identify the second summands, we have
$Q_2\simeq K_C\oplus K_C$ and
$Q_3\simeq K_C^2\oplus K_C^2$.  Hence
\begin{equation}\label{eq:lower-target}
 H^1\!\left(C,\cHom(Q_2,Q_3)\right)
 \simeq M_2\!\left(H^1(C,K_C)\right)
 \simeq M_2(\CC).
\end{equation}
The row and column order in this matrix space is the order
$(\Delta,\Gamma_\sigma)$ of the two quotient summands.

Use the Serre trace $H^1(C,K_C)\xrightarrow{\sim}\CC$ in
\eqref{eq:lower-target}.  Contraction of dual Serre bases with a canonical
extension tensor is then the ordinary matrix trace.  On the two
quotient summands, multiplication by $s\in V_1$ is
$\operatorname{diag}(s,\rho_1s)$.  Therefore the lower-right class of the
cocycle $Z_A$ is
\begin{equation}\label{eq:ZA-lr}
 [Z_A^{\mathrm{lr}}]
 =\operatorname{diag}\bigl(\tr A,\tr(A\rho_1)\bigr)
 =\operatorname{diag}(1,-1)
 \quad\text{in }M_2(\CC).
\end{equation}
On the other hand, write $R_1=I$ and $R_2=\rho_2^{-1}$ for the
two tensors of $e_2$.  For dual Serre bases $(q_k)$ and $(q_k^*)$, the
$(i,j)$ entry of $[\gamma e_2]$ is
$$
 \sum_k q_k^*(C_iR_jq_k)=\tr(C_iR_j).
$$
Thus composing \eqref{eq:C1C2} with the extension tensor
$e_2=(I,\rho_2^{-1})$ gives
\begin{equation}\label{eq:gamma-e2}
 [\gamma e_2]=
 \begin{pmatrix}
  \tr I_3&\tr\rho_2^{-1}\\
  -\tr\rho_2&-\tr I_3
 \end{pmatrix}
 =\operatorname{diag}(3,-3),
\end{equation}
because $1+\zeta+\zeta^2=0$.  Thus the residual lower-right class is
\begin{equation}\label{eq:final-matrix}
 [Z_A^{\mathrm{lr}}+\gamma e_2]
 =\begin{pmatrix}4&0\\0&-4\end{pmatrix}\neq0
 \quad\text{in }H^1\!\left(C,\cHom(Q_2,Q_3)\right).
\end{equation}
This contradicts the necessary identity
\eqref{eq:block-identity-lr}.  The lower-left, upper-left, and lower-right
blocks already give incompatible necessary conditions, so no equation
from the upper-right block is needed.  Consequently
\eqref{eq:solve-coboundary} has no solution, and
$H^1(j)(A)\neq0$.
\end{proof}

\begin{proof}[Proof of \cref{thm:main}]

By Propositions \ref{prop:explicit-nonzero} and \ref{prop:reduction}, the extension
\eqref{eq:jet-extension} obtained by restriction to $2\Delta_C$ is
non-split.  By Proposition \ref{prop:second-diagonal}, this extension is canonically
isomorphic to the categorical pushout of the invariant cross block
\eqref{eq:cross-block} along $\mu$.  It follows that the invariant cross
block is non-split, since a split extension would have split categorical
pushout.  By the normalized comparison in Lemma~\ref{lem:cross-block},
the $J$-invariant summand of \eqref{eq:DW} is therefore non-split.
The restriction-and-projection argument in the proof of
Proposition~\ref{prop:reduction} then shows that the full extension
\eqref{eq:DW} is non-split on the even spin family $T$.

All these constructions and extension classes are algebraic over
$\CC$.  Since $T$ is projective, coherent GAGA identifies
$H^1(T,G^\vee\otimes\bigwedge^2F)$ with the corresponding holomorphic
cohomology group and carries this extension class to its analytic class
\cite{SerreGAGA}.  Its non-vanishing therefore also rules out a
holomorphic projection.

Therefore
$$
 \kappa^*\omega_2(\M_4^+)\neq0
$$
by \cite[Theorem~3.2]{DonagiWittenAtiyah}.  Functoriality gives
$\omega_2(\M_4^+)\neq0$, so $\M_4^+$ is not projected.  A disjoint union
can be projected only if each component is; hence $\M_4$ is not projected.
\end{proof}

\end{document}